\documentclass[11pt]{amsart}

\usepackage{epigamath-voisin}

\usepackage[english]{babel}

\numberwithin{equation}{section}

\usepackage{amscd, amsmath, amssymb,amsfonts,bbm, calligra, xspace,color}
\usepackage{amsthm}
\usepackage[T1]{fontenc}
\usepackage{mathtools}
\usepackage{enumerate}
\usepackage{multicol}
\usepackage[all]{xy}
\usepackage{mathrsfs}
\usepackage{footmisc}

\usepackage[shortlabels]{enumitem}
\setlist[enumerate,1]{label={\rm(\alph*)}, ref={\rm(\alph*)}} 

\newtheorem{thm}{Theorem}[section]
\newtheorem{prop}[thm]{Proposition}

\newtheorem{lem}[thm]{Lemma}
\newtheorem{cor}[thm]{Corollary}

\theoremstyle{definition}

\newtheorem{quest}[thm]{Question}

\theoremstyle{remark}

\newtheorem{rems}[thm]{Remarks}

\newcommand{\Hdg}{\operatorname{Hdg}}

\newcommand{\td}{\operatorname{td}}
\newcommand{\alg}{\operatorname{alg}}
\newcommand{\ch}{\operatorname{ch}}
\newcommand{\MU}{\operatorname{MU}}
\newcommand{\BU}{\operatorname{BU}}
\def\alg{\operatorname{alg}}

\newcommand{\isoto}{\myxrightarrow{\,\sim\,}}

\makeatletter
\def\myrightarrow{{\setbox\z@\hbox{$\rightarrow$}\dimen0\ht\z@\multiply\dimen0 6\divide\dimen0 10\ht\z@\dimen0\box\z@}}
\def\myrightarrowfill@{\arrowfill@\relbar\relbar\myrightarrow}
\newcommand{\myxrightarrow}[2][]{\ext@arrow 0359\myrightarrowfill@{#1}{#2}}
\makeatother

\def\bS{\mathbb S}

\def\Z{\mathbb Z}
\def\C{\mathbb C}

\def\Q{\mathbb Q}

\YearArticle{2023} \EpigaArticleNr{2} \ReceivedOn{November 17, 2022}
\InFinalFormOn{February 24, 2023}
\AcceptedOn{March 16, 2023}

\title{Smooth subvarieties of Jacobians}
\titlemark{Smooth subvarieties of Jacobians}

\author{Olivier Benoist}
\address{D\'epartement de math\'ematiques et applications, \'Ecole normale sup\'erieure, CNRS,45 rue d'Ulm, 75230 Paris Cedex 05, France}
\email{olivier.benoist@ens.fr}

\author{Olivier Debarre}
\address{Universit\'e Paris Cit\'e and Sorbonne Universit\'e, CNRS, IMJ-PRG, F-75013 Paris, France}
\email{olivier.debarre@imj-prg.fr}

\authormark{O.~Benoist and O.~Debarre}

\AbstractInEnglish{We give new examples of algebraic integral cohomology classes on smooth projective complex varieties that are not integral linear combinations of classes of smooth subvarieties. Some of our examples have dimension $6$, the lowest possible. The classes that we consider are minimal cohomology classes on Jacobians of very general curves. Our main tool is complex cobordism.}



\acknowledgement{The second-named author has received funding from the European Research Council (ERC) under the European Union's Horizon 2020 research and innovation programme (Project HyperK---grant agreement~854361).}

\dedication{\`A Claire, pour son soixanti\`eme anniversaire} 

\begin{document}

 
\maketitle

\begin{prelims}

\DisplayAbstractInEnglish

\bigskip

\DisplayKeyWords

\medskip

\DisplayMSCclass

\end{prelims}


\newpage

\setcounter{tocdepth}{1}

\tableofcontents


\section{Introduction}

Let $c,d $ be nonnegative integers and let $X$ be a smooth projective complex variety of dimension $n:=c+d$. An important geometric invariant of $X$ is
the subgroup $H^{2c}(X,\Z)_{\alg}\subset H^{2c}(X,\Z)$ of algebraic cohomology classes, which is generated by the cycle classes of codimension~$c$ algebraic subvarieties $Y\subset X$. Since the subvarieties $Y $ may be singular, the following question going back to Borel and Haefliger \cite[Section~5.17]{BH} naturally arises.

\begin{quest}
\label{BH}
Is the group $H^{2c}(X,\Z)_{\alg}$ generated by classes of smooth subvarieties of $X$?
\end{quest}

The answer to Question~\ref{BH} is obviously positive if $d=0$ or if $c\leq 1$.\ Further positive answers were obtained by Hironaka \cite[Theorem, Section~5, ~p.~50]{Hironaka} when $d\leq\min(3,c-1)$ and by Kleiman \cite[Theorem~5.8]{Kleiman} when $c=2$ and $d\in\{2,3\}$; the answer is therefore positive  for all $n\le 5$.

A first counterexample was then constructed by Hartshorne, Rees, and Thomas \cite[Theorem~1]{HRT} when~$c=2$ and $d\geq 7$ (and their method should yield counterexamples for all~$c\geq 2$ and $d\gg c$). Other counterexamples were given by Debarre \cite[Th\'eor\`eme 6]{Debarre} when~$c=2$ and $d\geq 5$, and by Benoist \cite[Theorem 0.3]{approx} when $d\geq c$ and $\alpha(c+1)\geq 3$ (where $\alpha(m)$ denotes the number of ones in the binary expansion of~$m$). 

In this article, we build on the counterexample given in \cite{Debarre}. There, Debarre considers the Jacobian $X$ of a smooth projective complex curve $C$ of genus $n$, polarized by its theta divisor class $\theta\in H^2(X,\Z)$. The minimal cohomology class $\frac{\theta^c}{c!}\in H^{2c}(X,\Z)$ is the cycle class of the image $W_{n-c}(C)\subset X$ of the symmetric power $C^{(n-c)}$ by the Abel--Jacobi map,  hence is algebraic. However, when $n\ge 2c+2$, the variety $W_{n-c}(C)$ is singular and there is  no general reason why its class  $\frac{\theta^c}{c!}$ should be a $\Z$-linear combination of classes of smooth subvarieties of~$X$. Debarre shows that this is indeed not the case  if $n\geq 7$,~$c=2$, and  $X$ is very general.

 Our main theorem extends this result to many other values of $(c,n)$. Recall that $\alpha(m)$ is the number of ones in the binary expansion of~$m$.

\begin{thm}[Theorem~\ref{thsmooth}]
\label{main}
Let $(X,\theta)$ be a very general complex Jacobian of dimension~$n$. Let $c $ be a nonnegative integer such that $\alpha(c+\alpha(c))>\alpha(c)$ and $n\geq 4c-2$. Then the integral classes $\lambda\frac{\theta^c}{c!}$ with $\lambda$ odd are algebraic but are not $\Z$-linear combinations of cycle classes of smooth subvarieties of $X$.
\end{thm}

The weird condition $\alpha(c+\alpha(c))>\alpha(c)$ springs naturally from our proof. It holds for integers $c$ in the set $ \{2,4,5,8,9,12,16,17,\dots\}$.  We have nothing to say when $c=3$: we do not know if there exist Jacobians~$(X,\theta)$ such that $\frac{\theta^3}{3!}$ is not a $\Z$-linear combination of cycle classes of smooth subvarieties of $X$. Applied with~$c=2$ and $n=6$, Theorem~\ref{main} implies the following result.

\begin{cor}
\label{cormain}
There exists a smooth projective complex variety $X$ of dimension $6$ such that $H^4(X,\Z)_{\alg}$ is not generated by classes of smooth subvarieties of $X$.
\end{cor}

As noted above, the groups of algebraic cohomology classes of smooth projective varieties of dimension at most $5$ are generated by classes of smooth subvarieties. The  dimension of the variety $X$ in Corollary~\ref{cormain} is thus the lowest possible. We do not know, however, if Question~\ref{BH} has a positive answer if $c=d=3$. 

The proof of Theorem~\ref{main} when $c=2$ and $n\geq 7$ given in \cite{Debarre} relies on a Barth--Lefschetz-type theorem for abelian varieties, on the Hirzebruch--Riemann--Roch theorem, and on the Serre construction. We can still rely on the same Barth--Lefschetz-type theorem under our more general hypotheses. However,  computations based on the Hirzebruch--Riemann--Roch theorem become untractable in high codimension (and do not yield a proof when $c=2$ and $n=6$), and we cannot use the Serre construction in general because it is specific to codimension $2$ subvarieties. 

Both difficulties are overcome by resorting to complex cobordism, whose use in the theory of algebraic cycles was pioneered by Totaro \cite{Totaro}. We replace the integrality results derived from the Hirzebruch--Riemann--Roch theorem with divisibility properties of Chern numbers (Proposition~\ref{cong}). Those are obtained in Section~\ref{secChern} as a consequence of a detailed understanding of the structure of the complex cobordism ring. In \cite{Debarre}, the Serre construction was used to infer restrictions on the cohomology class of a smooth subvariety~$Y\subset X$. Instead, we consider the class of $Y$ in the complex cobordism of $X$ and use the description of the complex cobordism of an abelian variety (in Sections~\ref{parMU}--\ref{parMUab}). These tools are combined in Section~\ref{parproof} to prove Theorem~\ref{main}.

In low codimension, a different method based on complex topological $K$-theory and on the Grothendieck--Riemann--Roch theorem, closer to the one used in \cite{Debarre}, gives small improvements on Theorem~\ref{main}.  We use this method in Section~\ref{secCod4} to prove the following.

\begin{thm}[Theorem~\ref{thsmooth4}]
\label{main2}
Let $(X,\theta)$ be a very general complex Jacobian of dimension $n$. Then $\lambda\frac{\theta^4}{4!}$ is algebraic but not a $\Z$-linear combination of classes of smooth subvarieties of~$X$ 
\begin{enumerate}
\item\label{ffa} if $n\geq 12$ and $\lambda$ is odd;
\item\label{ffb} if $n\geq 14$ and $\lambda$ is not divisible by $4$.
\end{enumerate}
\end{thm}

\subsection*{Conventions} 
A complex variety is a separated scheme of finite type over~$\C$. If $M$ is a compact oriented manifold of dimension $d$, we let $[M]\in H_d(M,\Z)$ be the fundamental class of~$M$ and denote by $\deg_M\colon H^{d}(M,\Z)\to\Z$ the morphism ${\omega\mapsto \deg(\omega\frown [M])}$. We let $\alpha(m)$ denote the number of ones in the binary expansion of $m$.

\section{Congruences of Chern numbers}
\label{secChern}

\subsection{The Hurewicz morphism of $\mathbf{MU}$}
\label{Hurepar}

In this paragraph, we recall the structure of the Hurewicz morphism
\begin{equation}
\label{Hurewicz}
H\colon\pi_*(\mathbf{MU})\longrightarrow H_*(\mathbf{MU},\Z)
\end{equation}
of the spectrum $\mathbf{MU}$  representing complex cobordism. 

  The computation of the cohomology of complex Grassmannians and the Thom isomorphism combine to show that $H_*(\mathbf{MU},\Z)$ is a polynomial ring with integral coefficients with one generator in degree $2i$ for each $i\geq 1$ (see \cite[Section~I.3]{Adams}).
  A deep theorem of Milnor shows that $\pi_*(\mathbf{MU})$ is also a polynomial ring with integral coefficients with one generator in degree~$2i$ for each $i\geq 1$ and that $H$ is injective (see \cite[Theorem~II.8.1 and Corollary~II.8.11]{Adams}).
  Quillen's theorem identifying $\pi_*(\mathbf{MU})$ with the coefficient ring of the universal formal group law \cite[Theorem~II.8.2]{Adams} and Hazewinkel's explicit construction of a universal formal group law \cite{Hazewinkel} allow us to be more precise. 

\begin{prop}
\label{propHaze}
For each $i\ge1$, there exist $u_{i}\in \pi_{2i}(\mathbf{MU})$ and $v_i\in H_{2i}(\mathbf{MU},\Z)$ such that
\begin{enumerate}
\item $\pi_*(\mathbf{MU})=\Z[u_i]_{i\geq 1}$; 
\item  $H_*(\mathbf{MU},\Z)=\Z[v_i]_{i\geq 1}$;
\item $H(u_i)=\lambda_iv_i$, where $\lambda_i=p$ if $i=p^t-1$ for some $t\ge1$ and some prime number $p$, and $\lambda_i=1$ otherwise.
\end{enumerate}
\end{prop}

\begin{proof}
Let $u_i\in \pi_{2i}(\mathbf{MU})$ be the polynomial generators of $\pi_*(\mathbf{MU})$ specified in \cite[Section~34.4.1]{Hazewinkel}. Induction on $i$ using the formula \cite[(34.4.3)]{Hazewinkel} shows that there exist $v_i\in H_{2i}(\mathbf{MU},\Z)$ such that $H(u_i)=\lambda_iv_i$. It then follows from \cite[Lemmas~II.7.9(iii) and~II.8.10]{Adams} that $v_i$ generates $H_{2i}(\mathbf{MU},\Z)$ modulo its decomposable elements. Consequently, $H_*(\mathbf{MU},\Z)=\Z[v_i]_{i\geq 1}$.
\end{proof}

For each $e\geq 1$, define the ideal $I_e\subset \pi_*(\mathbf{MU})$ to be the kernel of the composition 
$$\pi_*(\mathbf{MU})\xrightarrow{\ H \ } H_*(\mathbf{MU},\Z)\xrightarrow{\ \ \ } H_*(\mathbf{MU},\Z)/2^e$$
of $H$ and  the reduction modulo $2^e$. Working in the monomial bases associated with the $u_i$ and the $v_i$ given by Proposition~\ref{propHaze} shows at once the following result.

\begin{lem}
\label{ideal}
One has $I_1=\langle 2, u_{2^t-1}\rangle_{t\geq 1}$ and $I_e=(I_1)^e$.
\end{lem}

\subsection{Chern numbers}
\label{Chernpar}

Let $\MU_*$ be the complex cobordism ring, whose degree $d$ elements are complex cobordism classes of $d$-dimensional compact stably almost complex manifolds. The Thom--Pontrjagin construction gives an identification \begin{equation}
\label{PT}
\xi\colon\MU_*\isoto \pi_*(\mathbf{MU})
\end{equation}
(apply \cite[Theorem 1.5.10]{Kochman} with $\mathfrak{B}=\BU$).

Consider the polynomial ring $\Z[c_j]_{j\geq 1}$ in the Chern classes $c_j$, graded so that $c_j$ has degree $j$. A degree~$i$ element $P\in \Z[c_j]_{j\geq 1}$ may be evaluated on a $2i$-dimensional compact stably almost complex manifold $M$ by setting $P(M)\coloneqq\deg_M(P(c_j(M)))$. 
This integer only depends on the complex cobordism class of $M$ (see \cite[Section~4.3, p.~135]{Kochman}, with $\mathfrak{B}=\BU$ and $E$ the ordinary integral cohomology spectrum), so we get a
 morphism $P\colon\MU_{2i}\to \Z$, called the Chern number associated with $P$. 
The following lemma is an immediate consequence of \cite[Proposition 4.3.8]{Kochman}.

\begin{lem}
\label{Chern}
Let $i\geq 0$. A morphism $\MU_{2i}\to\Z$ is a Chern number if and only if it may be written as $\psi\circ H\circ\xi$ for some group morphism $\psi\colon H_{2i}(\mathbf{MU},\Z)\to\Z$, where~$H$ and $\xi$ are as in~\eqref{Hurewicz} and \eqref{PT}.
\end{lem}

We define the Segre classes $s_i\in\Z[c_j]_{j\geq 1}$ to be the unique elements such that $s_i$ has degree~$i$ and  
\begin{equation}
\label{Segredef}
\left(\sum\nolimits_j c_j\right)\left(\sum\nolimits_i s_i\right)=1.
\end{equation}
The Chern numbers associated with $s_i$ have the following properties.

\begin{lem}\leavevmode
\label{Segre}\begin{enumerate}
\item\label{aa} If $x\in \MU_{2i}$ and $x'\in \MU_{2i'}$, then $s_{i+i'}(xx')=s_i(x)s_{i'}(x')$.
\item\label{ab} For $i\geq 0$ and $h\geq 1$, the function $s_i\colon\MU_{2i}\to\Z$ is divisible by $2^h$ if and only if $\alpha(i+h-1)>2h-2$.
\item\label{ac} For $i\geq 1$, the function $s_i\colon  \MU_{2i}\to\Z$ only takes even values and only takes values divisible by $4$ on decomposable elements.
\item\label{ad} For $i=2^t-1\geq 1$ and $u_i\in \MU_{2i}$ as in Proposition~\ref{propHaze},  $s_i(u_i)\equiv 2 \pmod 4$.\qedhere
\end{enumerate}
\end{lem}

\begin{proof}
For assertion~\ref{aa}, see \cite[Lemma 3.3]{approx}. Assertions~\ref{ab} and~\ref{ac} follow from a theorem of Rees and Thomas (\cite[Theorem~3]{RT}; see \cite[Theorem 3.4 and Corollary 3.5]{approx} for these exact statements). The same result of Rees and Thomas (\cite[Theorem~3]{RT} applied with~$r=0$ and $n=2^{t}-1$) shows that not all the values of $s_i\colon  \MU_{2i}\to\Z$ are divisible by $4$ if $i=2^t-1$. This fact combined with~\ref{ac} implies~\ref{ad}.
\end{proof}

\subsection{A congruence result for the top Segre class}

The next proposition is the main goal of this section. It simultaneously generalizes the theorem of Rees and Thomas recalled in Lemma~\ref{Segre}\ref{ab} (when $e=0$) and \cite[Theorem 3.6]{approx} (when $e=1$).

\begin{prop}
\label{cong}
Fix $e\geq 0$, $h\geq 1$, and $i\geq 1$. The following assertions are equivalent:
\begin{enumerate}[{\rm (i)}]
\item\label{ba} There exists a degree $i$ element $Q\in \Z[c_j]_{j\geq 1}$ such that the Chern number $$s_i+2^hQ\colon \MU_{2i}\longrightarrow\Z$$ only takes values divisible by $2^{e+h}$. 
\item\label{bb} One has $\alpha(i+e+h-1)>e+2h-2$.
\end{enumerate}
\end{prop}

\begin{proof}
Assertion~\ref{ba} implies that $s_i\colon \MU_{2i}\to \Z$ is divisible by $2^h$. So does assertion~\ref{bb} by Lemma~\ref{Segre}(b). We may thus assume that the function $s_i\colon \MU_{2i}\to \Z$ is indeed divisible by~$2^h$.

 Identify $\MU_{2i}$ and $\pi_{2i}(\mathbf{MU})$ using~\eqref{PT}. Consider the following statements:
\begin{enumerate}
\item\label{ca} The function $\frac{s_i}{2^h}\colon \MU_{2i}\to\Z$ coincides modulo $2^e$ with a Chern number.
\item\label{cb} The function $\frac{s_i}{2^h}\colon \MU_{2i}\to\Z/2^e$ factors through $H_{2i}(\mathbf{MU},\Z)$ via $H$.
\item\label{cc} The function $\frac{s_i}{2^h}\colon \MU_{2i}\to\Z/2^e$ factors through $H_{2i}(\mathbf{MU},\Z)/2^e$ via $H$.
\item\label{cd} The function $\frac{s_i}{2^h}\colon \MU_{2i}\to\Z/2^e$ vanishes on $(I_e)_{2i}$.
\item\label{ce} The function $s_i\colon \MU_{2i}\to\Z$ only takes values divisible by $2^{e+h}$ on $(I_e)_{2i}$.
\item\label{cf} If $i=(\sum_{k=1}^{e'}2^{t_k}-1)+j$ for some $e'\leq e$, then  $s_j\colon\MU_{2j}\to\Z$ is divisible by $2^h$.
\end{enumerate}

The equivalence~\ref{ba}$\Leftrightarrow$\ref{ca} is clear, and~\ref{ca}$\Leftrightarrow$\ref{cb} follows from Lemma~\ref{Chern}. The equivalence~\ref{cb}$\Leftrightarrow$\ref{cc} holds because $\MU_{2i}$ is $\Z$-free. The implication~\ref{cc}$\Rightarrow$\ref{cd} is a consequence of the definition of $I_e$, and the converse holds because $\Z/2^e$ is an injective $\Z/2^e$\nobreakdash-module \cite[Exercise~2.3.1]{Weibel}.  The equivalence~\ref{cd}$\Leftrightarrow$\ref{ce} is elementary. As for~\ref{ce}$\Leftrightarrow$\ref{cf}, it is a consequence of the description of $I_e$ given in Lemma~\ref{ideal} and of the properties of the Segre classes given in Lemma~\ref{Segre}\ref{aa} and~\ref{ad}. Finally, applying Lemma~\ref{Segre}\ref{ab} shows the equivalence~\ref{cf}$\Leftrightarrow$\ref{bb}.
\end{proof}

\section{Smooth cycles on abelian varieties}
\label{secCycles}

\subsection{Complex cobordism}
\label{parMU}

  We denote by $\MU_*(X)$ the complex cobordism homology theory
represented by the spectrum $\mathbf{MU}$, evaluated on a topological space $X$. The group $\MU_d(X)$ has a geometric interpretation as the group of complex cobordism classes of continuous maps $f\colon M\to X$, where $M$ is a $d$-dimensional compact stably almost complex manifold (see for instance \cite[Proposition 12.35]{Switzer}). We denote by $[f]\in \MU_d(X)$ the class of $f$. When $X$ is a point, one recovers the ring $\MU_*=\pi_*(\mathbf{MU})$ which we studied in Sections~\ref{Hurepar} and~\ref{Chernpar}. In general, the group $\MU_*(X)$ is naturally an $\MU_*$-module.

If $f\colon M\to X$ is as above and if $P\in\Z[c_j]_{j\geq 1}$ of degree $p$ and $\omega\in H^{d-2p}(X,\Z)$ are given,  the integer $\deg_M \big(P(c_j(M))\cdot f^*\omega\big)$ only depends on $[f]\in \MU_d(X)$ (the argument found in \cite[Section~17, p.~54]{Conner} for unoriented cobordism also works for complex cobordism). It is called the \textit{characteristic number} of $f$ associated with $P$ and $\omega$.

Fix a point $o\in\bS^1$. The compatibility of the homology theory $\MU_*$ with suspension shows that $\MU_*(\bS^1)$ is free of rank $2$ over $\MU_*$, generated by the classes of the inclusion $\{o\}\to\bS^1$ and of the identity $\bS^1\to\bS^1$ (where $\bS^1$ is endowed with the stably almost complex structure induced by a trivialization of its tangent bundle). 
For $N\geq 0$  and $E\subset\{1,\dots,N\}$, we consider the inclusion $f_E\colon T_E\hookrightarrow(\bS^1)^N$ of the subtorus $T_E\coloneqq\prod_{e\in E}\bS^1\times\prod_{e\in\{1,\dots,N\}\setminus E}\{o\}$.

\begin{lem}
\label{Kunneth}
Fix $N\geq 1$. When $E$ describes the set of all subsets of $\{1,\dots,N\}$,
\begin{enumerate}
\item\label{ka} the $f_{E,*}[T_E]\in H_{|E|}((\bS^1)^N,\Z)$ form an additive basis of $H_*((\bS^1)^N,\Z)${\rm ;}
\item\label{kb} the $[f_E]\in\MU_{|E|}((\bS^1)^N)$ form an $\MU_*$-basis of\, $\MU_*((\bS^1)^N)$.
\end{enumerate}
\end{lem}

\begin{proof}
The first assertion follows from the K\"unneth formula in ordinary homology and the second assertion from the K\"unneth formula in complex cobordism (apply \cite[Theorem~13.75~i)]{Switzer} with ${E=\mathbf{MU}}$).
\end{proof}

\begin{lem}
\label{characteristic}
Let $x\in\MU_{2i}$, and let $E\subset \{1,\dots,N\}$ be such that $|E|=d-2i$. Fix $P\in\Z[c_j]_{j\geq 1}$ of degree~$l$ and $\omega\in H^{d-2l}((\bS^1)^N,\Z)$. 
The characteristic number of $x\cdot[f_E]\in\MU_{d}((\bS^1)^N)$ associated with~$P$ and $\omega$ is equal to~$0$ if $l\neq i$ and to $P(x)\deg_{(\bS^1)^N}\big(f_{E,*}(1)\cdot\omega\big)$ if $l=i$.
\end{lem}

\begin{proof}
Let $M$ be a $2i$-dimensional compact stably almost complex manifold representing $x\in\MU_{2i}$. Let $g_E\colon M\times T_E\to X$ be the composition of the second projection and $f_E\colon T_E\to X$. 
We will also let $h_E\colon M\times T_E\to M$ denote the first projection. 
 Since the stably almost complex structure on the tangent bundle on the torus $T_E$ is stably trivial, one has $c_j(T_E)=0$ for $j>0$. It follows from the Whitney sum formula that $P(c_j(M\times T_E))=h_E^*P(c_j(M))\in H^{2l}(M\times T_E,\Z)$. As a consequence, $P(c_j(M\times T_E))\cdot g_E^*\omega\in H^d(M\times T_E,\Z)$ vanishes unless $l=i$. When $l=i$, we use the projection formula to compute 
 \begin{align*}
\deg_{M\times T_E}\left(P(c_j(M\times T_E))\cdot g_E^*\omega\right) & = \deg_{M\times T_E}\left(h_E^*P(c_j(M))\cdot g_E^*\omega\right) \\
&=\deg_{M}\left(P(c_j(M))\right)\deg_{T_E}\left(f_E^*\omega\right)\\
&  =P(x) \deg_{(\bS^1)^N}\left(f_{E,*}(1)\cdot\omega\right).\qedhere
\end{align*}
\end{proof}

\subsection{Complex cobordism of abelian varieties}
\label{parMUab}

Now let $X$ be a complex abelian variety of dimension $n$ with a principal polarization $\theta\in H^2(X,\Z)$. We identify $X$ and $(\bS^1)^N$ for $N=2n$ by means of a Lie group isomorphism $X\simeq (\bS^1)^{N}$. 
By Lemma~\ref{Kunneth}\ref{ka}, there exists for each~$k\geq 0$ a unique $\Z$-linear combination 
\begin{equation}
\label{tau}
\tau_k=\sum_{\substack{E\subset \{1,\dots,N\}\\|E|=2k}}\mu_E[f_E]\in \MU_{2k}(X)
\end{equation}
 such that $\sum_{|E|=2k}\mu_Ef_{E,*}[T_E]\in H_{2k}(X,\Z)$ is Poincar\'e-dual to the integral class $\frac{\theta^{n-k}}{(n-k)!}$ or,
in other words, such that
\begin{equation}
\label{mu}
\sum_{|E|=2k}\mu_Ef_{E,*}(1)=\frac{\theta^{n-k}}{(n-k)!}.
\end{equation}

\begin{prop}
\label{MUppav}
Let $(X,\theta)$ be a principally polarized complex abelian variety of dimension~$n$. Assume that the group $\Hdg^{2k}(X,\Z)$ of Hodge classes is generated by $\frac{\theta^{k}}{k!}$ for each $k\geq 0$, and let $\tau_k\in \MU_{2k}(X)$ be as in~\eqref{tau}. 
Let $f\colon Y\to X$ be a morphism of smooth projective varieties with~$Y$ of pure dimension~$d$. Then there exists, for each $i\in\{0,\dots,d\}$, an element $x_i\in\MU_{2i}$ such that 
\begin{equation}
\label{classinMU}
[f]=\sum_{i=0}^d x_{i}\cdot \tau_{d-i}\in \MU_{2d}(X). 
\end{equation}
\end{prop}

\begin{proof}
Let $R_i$ be the rank of the free $\Z$-module $\MU_{2i}$, and let $(y_{i,r})_{1\leq r\leq R_i}$ be a $\Z$\nobreakdash-basis of it. Since the $\MU_*$-module $\MU_*(X)$ is free with basis $([f_E])_{E\subset\{1,\dots,N\}}$ by Lemma~\ref{Kunneth}\ref{kb}, there exist unique integers $\nu_{i,r,E}$ such that
\begin{equation}
\label{triplesum}
[f]=\sum_{i=0}^d\sum_{r=1}^{R_i}\left( y_{i,r}\cdot\sum_{|E|=2d-2i}\nu_{i,r,E} [f_E]\right)\in\MU_{2d}(X).
\end{equation}

Fix $0\leq i\leq d$ and $1\leq r\leq R_i$. As $\MU_{2i}\isoto\pi_{2i}(\mathbf{MU})\xrightarrow{H} H_{2i}(\mathbf{MU},\Z)$ is an inclusion of free $\Z$-modules of the same rank $R_i$ by Proposition~\ref{propHaze}, it follows from Lemma~\ref{Chern} that there exists a degree $i$ element $P\in \Z[c_j]_{j\geq 1}$ such that $P(y_{i,s})$ is nonzero if and only if $s=r$. In view of Lemma~\ref{characteristic} and   the projection formula, the characteristic number of~\eqref{triplesum} associated with~$P$ and $\omega\in H^{2d-2i}(X,\Z)$ reads
\begin{equation*}
\deg_{X}\left(f_*(P(c_j(Y)))\cdot \omega\right)
=P(y_{i,r})\deg_{X}\left(\sum_{|E|=2d-2i}\nu_{i,r,E}\cdot f_{E,*}(1)\cdot\omega\right).
\end{equation*}
As $\omega$ is arbitrary, Poincar\'e duality on $X$ implies   
\begin{equation}
\label{Hodge}
f_*(P(c_j(Y)))=P(y_{i,r})\sum_{|E|=2d-2i}\nu_{i,r,E}f_{E,*}(1)\in H^{2n-2d+2i}(X,\Z).
\end{equation}
The left  side of~\eqref{Hodge} is algebraic, hence is a Hodge class. As a consequence, so is the right  side. Since $P(y_{i,r})\neq 0$, we deduce from our hypothesis that the class $\sum_{|E|=2d-2i}\nu_{i,r,E}f_{E,*}(1)$ is an integral multiple of $\frac{\theta^{n-d+i}}{(n-d+i)!}$. Equation~\eqref{mu} now implies   $\sum_{|E|=2d-2i}\nu_{i,r,E} [f_E]=\xi_{i,r}\tau_{d-i}$ for some~\mbox{$\xi_{i,r}\in\Z$.}

It finally follows from~\eqref{triplesum} that~\eqref{classinMU} holds with $x_i=\sum_{r=1}^{R_i}\xi_{i,r}y_{i,r}$.
\end{proof}

\begin{prop}
\label{Chernnumber}
Keep the hypotheses and notation of Proposition~\ref{MUppav}. 
Let $P\in\Z[c_j]_{j\geq 1}$ be of degree $l$ for some $0\leq l\leq d$. Then,
$$\deg_Y\left(P(c_j(Y))\cdot f^*\left(\frac{\theta^{d-l}}{(d-l)!}\right)\right)=\binom{n}{d-l}P(x_l).$$
\end{prop}

\begin{proof}
Let us compute the characteristic number of $f$ associated with $P$ and $\frac{\theta^{d-l}}{(d-l)!}$. 
Combining~\eqref{classinMU}, \eqref{tau}, and Lemma~\ref{characteristic} shows that it is
\begin{alignat*}{4}
\deg_Y\left(P(c_j(Y))\cdot f^*\left(\frac{\theta^{d-l}}{(d-l)!}\right)\right)&=P(x_l) \sum_{|E|=2d-2l}\mu_E\deg_{X}\left(f_{E,*}(1)\cdot\frac{\theta^{d-l}}{(d-l)!}\right)\\
&=P(x_l) \deg_{X}\left(\frac{\theta^{n-d+l}}{(n-d+l)!}\cdot\frac{\theta^{d-l}}{(d-l)!}\right).
\end{alignat*}
Since $\deg_{X}(\theta^n)=n!$, the proposition is proven.
\end{proof}

\subsection{Smooth subvarieties of abelian varieties} 
\label{parproof}

The next proposition is an application of a Barth--Lefschetz-type theorem proved by  Sommese \cite{Sommese}.

\begin{prop}
\label{BarthLefschetz}
Let $(X,\theta)$ be a principally polarized complex abelian variety of dimension~$n$ such that $\Hdg^{2k}(X,\Z)$ is generated by $\frac{\theta^{k}}{k!}$ for all $k\geq 0$. Let $f\colon Y\to X$ be the inclusion of a smooth projective subvariety of pure codimension $c$. Assume that $n\geq 4c-2l$ for some $l\geq 1$. Then there exist $a_0,\dots,a_{c-l},a_c\in\Z$ such that
\begin{enumerate}
\item\label{da} $s_i(Y)=a_if^*(\frac{\theta^i}{i!})$ for $i\in\{0,\dots,c-l,c\}${\rm ;}
\item\label{db} $f_*[Y]$ is Poincar\'e-dual to $a_c\frac{\theta^c}{c!}\in H^{2c}(X,\Z)$.
\end{enumerate}
\end{prop}

\begin{proof}
Since the subvariety $Y\subset X$ is algebraic, the homology class $f_*[Y]$ is Poincar\'e-dual to a Hodge class, which is necessarily of the form $a_c\frac{\theta^c}{c!}$ by hypothesis.\ This proves~\ref{db}.\ By the self-intersection formula \cite[Corollary~6.3]{Fulton}, the top Chern class of the normal bundle $N_{Y/X}$ is $a_cf^*(\frac{\theta^c}{c!})$.\ Since the tangent bundle $T_X$ is trivial, one has $c(N_{Y/X})=c(Y)^{-1}=s(Y)$.\ This shows~\ref{da} for $i=c$.

If the abelian variety $X$ were not simple, pulling back an ample divisor from a nontrivial quotient would produce a nonample divisor on $X$, contradicting the fact that $\Hdg^{2}(X,\Z)$ is generated by~$\theta$. We deduce that $X$ is simple. One may thus apply \cite[Corollary~3.5 and~(3.6.1)]{Sommese} with $B=Y$ to obtain $\pi_j(X,Y,y)=0$ for $j\leq n-2c+1$ and all $y\in Y$. It follows from the version \cite[Theorem~4.37]{Hatcher} of the Hurewicz theorem, from the universal coefficient theorem, and from the long exact sequence of relative cohomology of the pair~$(X,Y)$ that the restriction map $H^j(X,\Z)\to H^j(Y,\Z)$ is an isomorphism for $j\leq n-2c$. For $0\leq i\leq c-l$, one may apply this fact with $j=2i$ because $n\geq 4c-2l$. This shows that the class $s_i(Y)\in H^{2i}(Y,\Z)$, which is Hodge because it is algebraic,  is the restriction to $Y$ of a class in $\Hdg^{2i}(X,\Z)$. The latter is necessarily of the form $a_i\frac{\theta^i}{i!}$ by hypothesis. The proof is now complete.
\end{proof}

\begin{prop}
\label{even}
Keep the hypotheses and notation of Proposition~\ref{BarthLefschetz}, assume that $l=1$, and suppose in addition that $\alpha(c+\alpha(c))>\alpha(c)$. Then $a_c$ is even.
\end{prop}

\begin{proof}
Let $Q\in\Z[c_j]_{j\geq 1}$ be the degree $c$ homogeneous polynomial obtained by applying Proposition~\ref{cong} with $i=c$, $e=\alpha(c)$, and $h=1$. Applying Proposition~\ref{Chernnumber} with $l=c$ and $P=s_c+2Q$ yields the identity
\begin{equation}
\label{identity}
\deg_Y\left(P(c_j(Y))\cdot f^*\left(\frac{\theta^{n-2c}}{(n-2c)!}\right)\right)=\binom{n}{n-2c}P(x_c).
\end{equation}
Using that Chern classes may be expressed as polynomials with integral coefficients in Segre classes by~\eqref{Segredef}, it follows from Proposition~\ref{BarthLefschetz}\ref{da} that $Q(c_j(Y))\in H^{2c}(Y,\Z)$ is an integral multiple of $f^*(\frac{\theta^{c}}{c!})$, say $Q(c_j(Y))=b f^*(\frac{\theta^{c}}{c!})$ for some $b\in\Z$. Applying Proposition~\ref{BarthLefschetz}\ref{da} again, we get $P(c_j(Y))=(a_c+2b)f^*(\frac{\theta^{c}}{c!})$. Rewriting the left side of~\eqref{identity} using the projection formula and Proposition~\ref{BarthLefschetz}\ref{db}, we obtain
\begin{equation*}
\label{identity2}
\deg_X\left(a_c\frac{\theta^{c}}{c!}\cdot (a_c+2b)\frac{\theta^{c}}{c!}\cdot \frac{\theta^{n-2c}}{(n-2c)!}\right)=\binom{n}{n-2c}P(x_c).
\end{equation*}
Using   $\deg_X(\theta^n)=n!$, we finally get
\begin{equation}
\label{identity3}
P(x_c)=\binom{2c}{c}a_c(a_c+2b).
\end{equation}
Our choice of $Q$ implies that the left side of (\ref{identity3}) is divisible by $2^{\alpha(c)+1}$. The formula for the $2$\nobreakdash-adic valuation of the factorial given in \cite[Lemma, Section~5.3.1, p.~241]{Robert} implies that the $2$-adic valuation of $\binom{2c}{c}$ is equal to $\alpha(c)$. We deduce that $a_c(a_c+2b)$ is even, hence that $a_c$ is even.
\end{proof}

\begin{thm}
\label{thsmooth}
Let $(X,\theta)$ be a very general complex Jacobian of dimension $n$. Let $c\geq 0$ be such that $\alpha(c+\alpha(c))>\alpha(c)$ and $n\geq 4c-2$. Then the classes $\lambda\frac{\theta^c}{c!}$ with $\lambda$ odd are algebraic but are not $\Z$-linear combinations of cycle classes of smooth subvarieties of $X$.
\end{thm}

\begin{proof}
The integral class $\frac{\theta^c}{c!}$ is algebraic by \cite[Poincar\'e's formula  11.2.1]{BL}.

The hypothesis that $\Hdg^{2k}(X,\Z)$ is generated by $\frac{\theta^{k}}{k!}$ for all $k\geq 0$ is satisfied by \cite[Theorem~17.5.1]{BL} and because the integral class $\frac{\theta^{k}}{k!}$ is primitive. One may thus combine Propositions~\ref{BarthLefschetz}\ref{db} and~\ref{even} to show that the cycle classes of smooth codimension $c$ subvarieties $Y\subset X$ are even multiples of $\frac{\theta^c}{c!}$. This concludes the proof.
\end{proof}

\begin{rems}\leavevmode
\begin{enumerate}[(i),wide]
\item
In Theorem~\ref{thsmooth}, the hypothesis that $(X,\theta)$ is very general is only used to ensure that $\Hdg^{2k}(X,\Z)$ is generated by $\frac{\theta^{k}}{k!}$ for $k\geq 0$. 
As there exist Jacobians over $\overline{\mathbb{Q}}$ whose Mumford--Tate group is the full symplectic group (see \cite[Th\'eor\`eme 5.2 3) and Remarque~(vii) below it]{Andre} which apply as all Hodge classes on abelian varieties are absolute Hodge by \cite[Main Theorem 2.11]{Deligne}), one may find such an 
 $(X,\theta)$ that is defined over $\overline{\mathbb{Q}}$.

\item The proof of Theorem~\ref{thsmooth} actually shows that the class of any smooth subvariety of codimension~$c$ of~$X$ is divisible by $2$ in the group $H^{2c}(X,\Z)_{\alg}$.
\end{enumerate}
\end{rems}

\section{Codimension 4 cycles}
\label{secCod4}

Theorem~\ref{thsmooth} is not optimal in several respects. When $c$ is fixed, it is sometimes possible to give stronger restrictions on the cycle classes of smooth subvarieties of codimension $c$ of~$X$ or results for lower values of the dimension $n$ of $X$. 

To obtain such improvements, one may work with complex topological $K$-theory instead of complex cobordism   and replace the divisibility result for Chern numbers given in Proposition~\ref{cong} with an application of the Grothendieck--Riemann--Roch theorem and an integrality property of the Chern character (\textit{cf.}~Lemma~\ref{Chernint}). 
This works well when $c$ is low but becomes intractable for high values of $c$. We illustrate the method when~$c=4$.

We start with a lemma. Let $X$ be a topological space, and let
$K^*(X)=K^0(X)\oplus K^1(X)$ be its $\Z/2$-graded complex topological $K$-theory   defined in \cite[Section~1.9]{AHKtheory}. We consider the Chern character ${\ch:K^*(X)\to H^*(X,\Q)}$ as in \cite[Section~1.10]{AHKtheory}.

\begin{lem}
\label{Chernint}
For $N\geq 1$, the Chern character $\ch:K^*((\bS^1)^N)\to H^*((\bS^1)^N,\Q)$ is an isomorphism onto $H^*((\bS^1)^N,\Z)$.
\end{lem}

\begin{proof}
  First, suppose that $N=1$. The morphism $\ch:K^0(\bS^1)\to H^0(\bS^1,\Q)=\Q$ has image $\Z$ because it associates with each vector bundle its rank, and it is injective because $K^0(\bS^1)\simeq\Z$. That $\ch:K^1(\bS^1)\to H^1(\bS^1,\Q)$ is an isomorphism onto $H^1(\bS^1,\Z)$ follows from the definition of this morphism using suspension and Bott periodicity (see \cite[Section~1.10]{AHKtheory}) and from the fact the Chern character sends the Bott element which is a generator of $\widetilde{K}^0(\bS^2)$ to a generator of $H^2(\bS^2,\Z)$ (see \cite[Section~1.10, p.~206]{AHKtheory}).
 
It now follows from the K\"unneth formula in cohomology and in complex topological $K$\nobreakdash-theory (for which see \cite[Lemma 1]{AKunneth}) and from the multiplicativity of the Chern character that $\ch:K^*((\bS^1)^N)\to H^*((\bS^1)^N,\Q)$ is an isomorphism onto $H^*((\bS^1)^N,\Z)$.
\end{proof}

\begin{prop}
\label{codim4}
Let $(X,\theta)$ be a principally polarized complex abelian variety of dimension~$n$ such that $\Hdg^{2k}(X,\Z)$ is generated by $\frac{\theta^{k}}{k!}$ for all $k\geq 0$. Let $Y\subset X$ be a smooth projective subvariety   of pure codimension $4$. 
\begin{enumerate}
\item\label{ea} If $n\geq 12$,   the cohomology class of $Y$ is an integral multiple of $2\frac{\theta^4}{4!}$.
\item\label{eb} If $n\geq 14$,   the cohomology class of $Y$ is an integral multiple of $4\frac{\theta^4}{4!}$.
\end{enumerate}
\end{prop}

\begin{proof}
Let ${f:Y\to X}$ be the inclusion map.
 Proposition~\ref{BarthLefschetz} shows the existence of integers~$a_i $ such that $s_i(Y)=a_if^*(\frac{\theta^i}{i!})$ for $i\in\{0,1,2,4\}$ (if $n\geq 12$), or $i\in\{0,1,2,3,4\}$ (if $n\geq 14$), and such that, in addition, the cohomology class of $Y$ in $X$ is equal to $a_4\frac{\theta^4}{4!}$. As the class $f_*s_3(Y)$ is Hodge, we may also write $f_*s_3(Y)=b\frac{\theta^{7}}{7!}$ for some $b\in\Z$.

As noted during the proof of Proposition~\ref{BarthLefschetz}, one has $c(-N_{Y/X})=s(Y)^{-1}$,
hence
\begin{alignat*}{2}
c(-N_{Y/X})&=1-a_1f^*\theta+(a_1^2-a_2/2)f^*\theta^2+((a_1a_2-a_1^3)f^*\theta^3-s_3(Y))\\
&\hspace{.9em}+((a_1^4-3a_1^2a_2/2+a_2^2/4-a_4/24)f^*\theta^4+2a_1s_3(Y)f^*\theta)+\cdots.
\end{alignat*}
One can then compute the Todd class (see \cite[Example 3.2.4]{Fulton}):
\begin{equation}
\label{Td}
\begin{alignedat}{2}
\td(-N_{Y/X})&=1-a_1f^*\theta/2+(4a_1^2-a_2)f^*\theta^2/24+(a_1a_2-2a_1^3)f^*\theta^3/48\\
&\hspace{.9em}+(144a_1^4-108a_1^2a_2+12a_2^2+a_4)f^*\theta^4/17\,280-a_1s_3(Y)f^*\theta/720+\cdots.
\end{alignedat}
\end{equation}

Let $\chi\in K^0(X)$ be the image in complex topological $K$-theory of the algebraic $K$-theory class $f_*[\mathcal{O}_Y]$. The Grothendieck--Riemann--Roch theorem \cite[Theorem 15.2]{Fulton} applied to the inclusion $f$ and to the algebraic $K$-theory class $[\mathcal{O}_Y]$ shows that $\ch(\chi)=f_*\td(-N_{Y/X})$. Since~$f_*f^*$ is the cup-product with the cohomology class of~$Y$ which equals $a_4\frac{\theta^4}{4!}$, this identity may be combined with (\ref{Td}) to give
\begin{equation}
\label{ch}
\begin{alignedat}{2}
\ch(\chi)&=a_4\theta^4/24-a_1a_4\theta^5/48+(4a_1^2-a_2)a_4\theta^6/576+(a_1a_2-2a_1^3)a_4\theta^7/1\,152\\
&\hspace{.9em}+(144a_1^4-108a_1^2a_2+12a_2^2+a_4)a_4\theta^8/414\,720-ba_1\theta^8/3\,628\,800+\cdots.
\end{alignedat}
\end{equation}
By Lemma~\ref{Chernint}, which applies because $X$ is diffeomorphic to $(\bS^1)^{2n}$, the left side $\ch(\chi)$ of (\ref{ch}) is an integral cohomology class, hence so is its right side.
 
Assume by way of contradiction that $a_4$ is odd. The integrality of the class $\ch_{5}(\chi)=a_1a_4\theta^5/48$ implies that $a_1$ is even, and the integrality of the class $\ch_{6}(\chi)=(4a_1^2-a_2)a_4\theta^6/576$ shows that $a_2$ is divisible by~$4$. All terms  appearing in the expression for $\ch_{8}(\chi)$ given in (\ref{ch}) are multiples of~$\frac{\theta^8}{8!}$ with coefficients a rational number of nonnegative $2$-adic valuation, with the exception of $a_4^2\theta^8/414\,720=\frac{7a_4^2}{72}\frac{\theta^8}{8!}$ since $a_4$ is assumed to be odd. It follows that $\ch_{8}(\chi)$ is not an integral cohomology class, which gives a contradiction. This proves~\ref{ea}.

Now suppose $n\geq 14$. Then
$b\theta^{7}/7!=f_*s_3(Y)=f_*f^*(a_3\theta^3/3!)=a_3a_4\theta^7/144$, so that $b=35a_3a_4$.  Assume by way of contradiction that $a_4$ is not divisible by $4$. By the integrality of $\ch_{5}(\chi)=a_1a_4\theta^5/48$, we see that $a_1a_4$ is even. The integrality of $\ch_{6}(\chi)=(4a_1^2-a_2)a_4\theta^6/576$ shows that $a_2$ is even. These pieces of information imply that in the right side of the equality
$$\ch_{8}(\chi)=(144a_1^4-108a_1^2a_2+12a_2^2+a_4)a_4\theta^8/414\,720-35a_1a_3a_4\theta^8/3\,628\,800,$$
all terms are multiples of~$\frac{\theta^8}{8!}$ with coefficients a rational number of nonnegative $2$-adic valuation, with the exception of $a_4^2\theta^8/414\,720=\frac{7a_4^2}{72}\frac{\theta^8}{8!}$ since we assumed that~$a_4$ is not a multiple of $4$. This contradicts the integrality of $\ch_{8}(\chi)$ and proves~\ref{eb}.
\end{proof}

Combining Proposition~\ref{codim4} with \cite[11.2.1 and Theorem 17.5.1]{BL}, we get the following.

\begin{thm}
\label{thsmooth4}
Let $(X,\theta)$ be a very general complex Jacobian of dimension $n$. For $\lambda\in\Z$, the class $\lambda\frac{\theta^4}{4!}$ is algebraic but is not a $\Z$-linear combination of classes of smooth subvarieties of $X$ 
\begin{enumerate}
\item\label{fa} if $n\geq 12$ and $\lambda$ is odd;
\item\label{fb} if $n\geq 14$ and $\lambda$ is not divisible by $4$.
\end{enumerate}
\end{thm}


\providecommand{\bysame}{\leavevmode\hbox to3em{\hrulefill}\thinspace}
\providecommand{\MR}{\relax\ifhmode\unskip\space\fi MR }
\providecommand{\MRhref}[2]{%
  \href{http://www.ams.org/mathscinet-getitem?mr=#1}{#2}
}
\providecommand{\href}[2]{#2}

\end{document}